\documentclass[a4,12pt,reqno]{amsart}
\usepackage{amsmath,amsthm,amssymb}

\allowdisplaybreaks
\usepackage{color}

\theoremstyle{plain}
\newtheorem{thm}{Theorem}[section]

\newtheorem{prop}[thm]{Proposition}

\newtheorem{defn}[thm]{Definition}
\newtheorem{example}[thm]{Example}
\newtheorem{remark}[thm]{Remark}
\newtheorem{assu}[thm]{Assumption}
  
  \makeatletter
  \@addtoreset{equation}{section}
  \makeatother

\title[Ehrenfest urns]
{Interactions between Ehrenfest's urns arising from group actions}
\author[Mizukawa]{Hiroshi Mizukawa}
\thanks{The author was supported by KAKENHI 15K04802.}
\address{Hiroshi Mizukawa, Department of Mathematics,  National Defense
Academy of Japan,  Yokosuka 239-8686, Japan}
\email{mzh@nda.ac.jp}
\date{}
\keywords{ Gelfand pair of finite groups, Multivariate Krawtchouk polynomial, Cut-off phenomenon}
\subjclass[2010]{Primary: 05E18 ; Secondary: 60C05 }

\begin{document}
\maketitle
\begin{abstract}
Ehrenfest's diffusion model is a well-known classical physical model consisting of two urns 
and $n$ balls. 
A group theoretical interpretation of the model by using the Gelfand pair
$({\mathbb Z}/2{\mathbb Z}\wr S_{n},S_{n})$ is provided by Diaconis-Shahshahani \cite{ds}. 
This  interpretation remains valid for an $r$-urns generalization, in which case,
 the corresponding Gelfand pair  is $(S_{r}\wr S_{n},S_{r-1}\wr S_{n})$.
In these models, there are no restrictions for ball movements, i.e., 
each balls can freely move to any urn. 
This paper introduces interactions between urns arising from actions of finite groups.
It gives a framework of analysis of urn models whose interaction between urns is given by group actions.
The degree of freedom of ball movements is restricted by finite groups actions. 
Furthermore, for some cases, the existence of 
the cut-off phenomenons is shown.
\end{abstract}
\section{Introduction}\label{intro}

Ehrenfest's diffusion model \cite{feller} is a stochastic process consisting of  $n$-balls
$\{b_{1},\ldots,b_{n}\}$ and two urns $\{U_{0},U_{1}\}$. The model is constructed in the following way:
Let ${\mathbb Z}/2{\mathbb Z}=\{0,1 \}$ be a cyclic group of degree 2.
We set $U(2,n)=\{(x_{1},\ldots,x_{n}) \mid  x_{i}\in {\mathbb Z}/2{\mathbb Z}\}$.
We identify an element $(x_{1},\ldots,x_{n})$  of $U(2,n)$ with a state indicating whether  
the ball $b_i$ is in the urn $U_{x_{i}}$ $(x_{i}=0 \ {\rm or}\ 1)$.
We define a stochastic matrix on $U(2,n)$ by
$$p(x,y)=
\begin{cases}
1/(n+1) & x=y,\\
1/(n+1) & d(x,y)=1,\\
0 & otherwise .
\end{cases}$$
Here, $d(x,y)=|\{i \mid x_{i}\not= y_{i},\ 1 \leq i \leq n\}|$ is the Hamming distance on 
$U(2,n)$. The second condition above means that each ball moves to another urn under a probability 
$1$ in $n+1$ at each step.  We can consider a generalization of the above-mentioned
 stochastic process above by increasing the number of urns.  
Let $U(r,n)=\{(x_{1},\ldots,x_{n}) \mid  x_{i}\in S_{r}/S_{r-1}\}$.
Here, $S_{r}$ is the symmetric group on $\{0,1,\cdots,r-1\}$ and $S_{r-1}
=\{\sigma \in S_{r} \mid \sigma(r-1)=r-1\}$ is 
a subgroup of $S_{r}$. 
Then, we define a stochastic matrix on $U(r,n)$
by $$p(x,y)=
\begin{cases}
\frac{1}{n+1} & x=y,\\
\frac{1}{(r-1)(n+1)} & d(x,y)=1,\\
0 & otherwise .
\end{cases}$$
Note that setting $r=2$ gives the original 2-urn case is recovered.
In this generalization, note that each ball  move to any urn in each step.
 
 Diaconis and Shahshahani \cite{ds} analyzed the asymptotic behavior of the 
 $N$-step probability of the original 
 Ehrenfest's diffusion model.
In particular, they showed that the cut-off phenomenon occurs in this case.
 For general $r$-urns cases, Hora \cite{hora1,hora2} gave precise analysis of the asymptotic behavior of the $N$-step probability.
 He also showed that the cut-off phenomenon occur in these cases.
 
 In these cases, each model is realized by certain finite homogenous space arising from 
 a finite Gelfand pair $(S_{r}\wr S_{n},S_{r-1}\wr S_{n})$ $(n=2,3,4,\ldots)$.
We remark that the theory of finite Gelfand pair works well in these cases.
One can find  detailed descriptions of the theory of Gelfand pairs 
in Macdonald's book \cite{mc} and its relation to probability theory in \cite{cst}.
However, in these cases, balls can freely move to any urns at each step, or in other words,
there are no interactions between urns. Of course, many other interactions could be
considered.  In this paper, we give some interactions between urns through actions of 
finite groups actions.
It gives a framework of analysis of urn models whose interaction between urns is given by group actions.

The main tool considered in this study
 is  a multivariate hypergeometric type orthogonal polynomial called a multivariate Krawtchouk
polynomial \cite{gr1,mz}. 
An important feature is that the Krawtchouk polynomials can express the zonal spherical functions of a 
 finite Gelfand pair $(K \wr S_{n},L \wr S_{n})$ \cite{mt}. 
Multivariate Krawtchouk polynomials are known to have
 an application in the  stochastic model of the poker dice game \cite{gr}.
In this paper, we present another stochastic theoretical application of the multivariate Krawtchouk polynomials, that is, to the  Ehrenfest's diffusion models.

The remainder of this paper is organized as follows. 
 In section 2, we construct interaction urn models by using a finite Gelfand pair $(K,L)$ and 
 prepare some tools  to analyze our models. 
  In section 3, we compute upper and lower bounds of the total variation distance between
 an $N$-step probability and the uniform distribution on our models. 
As a conclusion of our discussion, we show that our models have the cut-off phenomenon.

\section{Preliminaries}
Let $K$ be a finite group and $L$ be its subgroup.
Set $e_{L}=\frac{1}{|L|}\sum_{h \in L}h$ which is an element of the group 
algebra ${\mathbb C}K$. Then the following  conditions are eqivalent:
(1) The Hecke algebra ${\mathcal H}(K,L)=e_{L}{\mathbb C}K
e_{L} \subset {\mathbb C}K$  is commutative.
(2) The permutation representation ${\mathbb C}Ke_{L}$ is multiplicity-free as $K$-module.
These condition are satisfied,
the pair $(K,L)$ is called a {\it Gelfand pair}. 
Here we remark that ${\mathcal H}(K,L)$ can be identified with the ring
of bi-$L$ invariant functions, i.e. functions on double coset $L\backslash 
K/L$, on $K$. 
Assume from now that $(K,L)$ is a Gelfand pair and
multiplicity-free decomposition of the permutation representation is given by
${\mathbb C}e_{L}=\bigoplus_{i=0}^{s-1} V_{i}$.
It is fact that $s=\dim {\mathcal H}(K,L)$. 
Let $\chi_{i}$ be the character of $V_{i}$. We define function $\omega_{i}$ by
$\omega_{i}(g)=\frac{1}{|L|}\sum_{h \in L}\overline{\chi_{i}(gh)}$.
The functions $\omega_{i}\ (0 \leq i \leq s-1)$ are called {\it zonal spherical functions}
 of Gelfand pair $(K,L)$. The zonal spherical functions is a  basis of ${\mathcal H}(K,L)$
 satisfying orthogonality relations
$$\frac{1}{|K|}\sum_{g\in K}
\omega_{i}(g){\overline {\omega_{j}(g)}}=
\delta_{ij}\dim V_{i}^{-1}.$$
Therefore we can expand a bi-$L$ invariant function  on $K$ by using the 
zonal spherical functions. 
One can find more detailed descriptions of the theory of Gelfand pairs 
in Macdonald's book \cite{mc}

Let $K$ be a finite group. 
 The wreath product $K \wr S_{n}=\{(g_{1},\ldots,g_{n};\sigma)\mid
 g_{i}\in K,\ \sigma \in S_{n}\}$ is
the semidirect product of $K^{n}$ with $S_{n}$  whose product is defined by 
$$(g_{1},\ldots,g_{n};\sigma)(g'_{1},\ldots,g'_{n};\sigma')=(g_{1}g'_{\sigma^{-1}(1)},\ldots,g_{n}g'_{\sigma^{-1}(n)};\sigma\sigma').$$
Let a positive integer $t$ be the number of conjugacy classes of $K$. 
The irreducible representations of $K \wr S_{n}$ is determined by the
$t$-tuples
 of partitions $(\lambda^{0},\ldots,\lambda ^{t-1})$ such that $|\lambda^{0}|+\cdots+|\lambda^{t-1}|=n$ (see \cite{jam}).

The next proposition is the fundamental result of \cite{mt}.
\begin{prop}[\cite{mt}]\label{zsf}
If $(K,L)$ is a Gelfand pair, then $(K \wr S_{n},L \wr S_{n})$ is also a Gelfand pair.
The double coset 
and the zonal spherical functions of $(K \wr S_{n},L \wr S_{n})$ are
 parameterized by $X(s,n)=\{(k_{0},\ldots,k_{s-1}) \mid k_{0}+\cdots +k_{s-1}=n\}$.
The values of the zonal spherical functions $\Omega_{k}$
at $x=(x_{1},\ldots,x_{n};\sigma) \in K \wr S_{n}$
 are given by the coefficients of $t^{k}$ in
$$\Phi_{(K,L;n)}(x)=\prod_{i=1}^n\left(\sum_{j=0}^{s-1}\omega_{j}(x_{i})t_{j}\right)=\sum_{k \in X(s,n)}
\binom{n}{k}\Omega_{k}(x)t^k.$$
Here, $\binom{n}{k}=\binom{n}{k_{0},\ldots,k_{r-1}}$ for $k=(k_{0},\ldots,k_{r-1}) \in X(s,n)$.
 \end{prop}
 The functions $\Omega_{k}$ in Proposition \ref{zsf} are called 
 {\it the multivariate Krawtchouk polynomials} (\cite{gr1,gr, mz}).

\section{Formulation of interaction urn models arising from group action}
Throughout this paper,  a pair of finite groups
$(K,L)$ is a Gelfand pair. 
For the irreducible decomposition of permutation representation ${\mathbb C}Ke_{L}=\bigoplus_{i=0}^{s-1}V_{i}$,
let $V_{0}$ be the trivial representation of $K$. Consider $d_{i}=\dim V_{i}\ (0 \leq i \leq s-1)$. 
Let $\omega_{i}$ be a zonal spherical function of $(K,L)$ afforded by $V_{i}$.

Fix an element $x_{0} \in K- L$. 
Set $D_{0}=L x_{0} L$.
 Put $$\begin{cases}
  r=\frac{|K|}{|L|},\\
 m=\frac{|D_{0}|}{|L|}.
 \end{cases}$$  
For $x\in K^n$, we denote by $\overline{x}$ a left coset $xL^n \in K^n/L^n$. 
%
\begin{defn}\label{p}
Let $0\leq mp \leq1$.
\begin{enumerate}
\item[(1)]
We define a function on $(K/L)^n \times (K/L)^n$ by
$$p(\overline{x},\overline{y})=\begin{cases}
1-mp & x^{-1}y \in L^{n}\\
p/n& x^{-1}y \in L^{k}\times D_{0} \times L^{n-k-1}\\
0 & {\rm otherwise}.
\end{cases}$$
\item[(2)]
We define a matrix $P_{(K,L;n)}$ by $$P_{(K,L;n)}=(p(\overline{x},\overline{y}))_{\overline{x},\overline{y}\in (K/L)^n \times (K/L)^n}.$$
\end{enumerate}
\end{defn}
In the above definition, we remark two points.
First the conditions $x^{-1}y \in L^n$ and $ x^{-1}y\in L^{k}\times D_{0} \times L^{n-k-1}$ do not depend on the choice of a complete set of 
representatives of $(K/L)^n$. 
Second the random walk depends on the choice of $x_{0}$.
Here, if we  
identify $(K/L)^n=\{(x_{1},\ldots,x_{n})\mid x_{i}\in K/L\}$ with the states  
$U(r,n)=\{(x_{1},\ldots,x_{n})\mid 0\leq x_{i} \leq r-1\}$ as a set,
we can interpret  Definition \ref{p} as a certain stochastic process of $n$-balls and $r$-urns.
The interaction between urns is determined by a Gelfand pair $(K,L)$. 
To understand the meaning of the interpretation, we provide some examples. 
%
\begin{example}
In the below examples, we define $d(\overline{x},\overline{y})=|\{i\mid x_{i}\not=y_{i}\}|$
for $\overline{x}=(\overline{x_{1}},\ldots,\overline{x_{n}})$ and 
 $\overline{y}=(\overline{y_{1}},\ldots,\overline{y_{n}})$.
 \begin{enumerate}\label{shut}
\item[(1)] Let $K=S_{r}$ and $L=S_{r-1}$. Here, $S_{r-1}$ is a subgroup of $S_{r}$ generated by 
transpositions $\{(i,i+1)\mid 1\leq i \leq r-2\}$.
A complete system of representatives for the left cosets and that for the double cosets of $L$
in $K$ are given by  
$$K/L=\{e_{K},(1,r),(2,r),\ldots,(r-1,r)\}\ \ {\rm and}\ \ L\backslash K/L=\{e_{K},(1,r)\}.$$
Let $x_{0}=(1,r)$. Then, we have $|D_{0}|=(r-1)(r-1)!$ and $m=(r-1)$.
We have $(i,r)(j,r) \in 
\begin{cases}
 D_{0}& i \not=j \\
L& i=j.
\end{cases}$
Now, we identify a left coset $(i,r)L$ with an urn $U_{i}$ $(0 \leq i \leq r-1)$. 
If we set $p=\frac{n}{(r-1)(n+1)}$, then $P_{(K,L;n)}$ is the same matrix of an $r$-urns generalization of  Ehrenfest's diffusion model introduced in Section \ref{intro}.
\item[(2)] Let  $K=\langle a\mid a^r=e  \rangle$ be a cyclic group and
its subgroup $L=\{e\}$.  
Setting $x_{0}=a$,
we have $m=|LaL|/|L|=1$.
Then,
$$a^{-i}a^{j} \in D_{0}\Leftrightarrow 
j-i\equiv
1 \pmod{r}.$$
Set $\overline{x}=x=(a^{i_{1}},\ldots,a^{i_{n}})$ and 
$\overline{y}=y=(a^{j_{1}},\ldots,a^{j_{n}})$.
Then,
the $(\overline{x},\overline{y})$-elements of $P_{(K,L;n)}$ are given by
$$
p(\overline{x},\overline{y})
=\begin{cases}
1-p & \overline{x}=\overline{y}\\
p/n& d(\overline{x},\overline{y})=1\ and\  j_{m}-i_{m} \equiv{0\ {\rm or\ 1}} \pmod{r}\ \\
0 & otherwise.
\end{cases}
$$
This case can be interpreted as  an $n$-balls and $r$-urns  model as follows:
At each step, 
 a ball is randomly chosen with the probability $1/n$
 and the chosen ball, which is in urn $U_{i}$,
 either  moves only to its left neighbor urn $U_{i+1}$   $(i,i+1 \in {\mathbb Z}/r{\mathbb Z})$ 
  with probability $p$
 or stays in the same urn with probability $1-p$.
\item[(3)] We consider a dihedral group
$K=D_{r}=\langle a,b \mid a^{r}=b^2=(ab)^2=1 \rangle$ and its subgroup $L=\langle b\rangle$.
A complete representatives of $L\backslash K/L$ is given by 
$\{1,a,a^2,\cdots,a^{[\frac{r}{2}]}\}$.
Set $x_{0}=a$. We have $m=|LaL|/|L|=2$. 
Set $\overline{x}=(a^{i_{1}},\ldots,a^{i_{n}})$ and 
$\overline{y}=(a^{j_{1}},\ldots,a^{j_{n}})$.
Since $a^{-i}a^{j}\in LaL \Leftrightarrow j \equiv i \pm1 \pmod{r}$,
the $(\overline{x},\overline{y})$-elements of $P_{(K,L;n)}$ are given by
$$p(\overline{x},\overline{y})=\begin{cases}
1-2p & \overline{x}=\overline{y}\\
p/n& d(\overline{x},\overline{y})=1\ {\rm and}\ \ j_{k}-i_{k}\equiv 0\ {\rm or}\ \pm 1 \pmod{r}\\
0 & otherwise.
\end{cases}$$
This case can be interpreted  as an $n$-balls and $r$-urns  model as follows:
At each step, 
 a ball is randomly chosen with the probability $1/n$
 and the chosen ball, which is in urn $U_{i}$,
  moves only to either of its both neighbor urns $U_{i-1}$ and $U_{i+1}$
$(i,i\pm1 \in {\mathbb Z}/r{\mathbb Z})$
  with probability $p$
 or stays in the same urn with probability $1-2p$.
$(i,i\pm1 \in {\mathbb Z}/r{\mathbb Z})$.
\item[(4)] In general, $(K\times K,\Delta K)$ is a Gelfand pair for any finite group $K$.
Here $\Delta K=\{(g,g)\mid g \in K\}$ is a diagonal subgroup of $K \times K$.
This fact supports that our model includes numerous examples. 
\end{enumerate}
\end{example} 
%
From these examples, we see that 
 $m$ gives the number of directions
  of ball movements.

 We define an action of $K \wr S_{n}$ on $(K/L)^n$ by
$$(g_{1},\ldots,g_{n}:\sigma)(\overline{x_{1}},\ldots,\overline{x_{n}})=
(\overline{g_{1}x_{\sigma(1)}},\ldots,\overline{g_{n}x_{\sigma(n)}}),$$
where $\overline{x_{i}} =x_{i}L$ ($x_{i} \in K$). 
From this definition, we see that $K \wr S_{n}$ acts on $(K/L)^n$ transitively. 
 \begin{prop}\label{trp}
 $P_{(K,L; n)}$ is  a $K \wr S_{n}$-invariant stochastic matrix.
\end{prop}
\begin{proof}
It is clear that $p(g\overline{x},g\overline{y})=p(\overline{x},\overline{y})\geq 0$ from Definition \ref{p}.
Let $e_{K}$ be the identity element of $K$ and let $\overline{e}=(\overline{e_{K}},\ldots,\overline{e_{K}}) \in (K/L)^n$.
We compute 
\begin{align*}
\sum_{y \in (K/L)^n}p(\overline{x},\overline{y})&=\sum_{y \in (K/L)^n}p(\overline{e},\overline{x^{-1}y})\\
&=\sum_{y \in (K/L)^n}p(\overline{e},\overline{y})\\
&=p(\overline{e},\overline{e})+\sum_{i=1}^n\sum_{\overline{y} \in \{\overline{e_{K}}\}^{i-1}\times D_{0}/L \times \{\overline{e_{K}}\}^{n-i}}
p(\overline{e},\overline{y})\\
&=(1-mp)+n(mp/n)=1.
\end{align*}
\end{proof}
 
Thus, with the natural identification between $(K/L)^n$ and $K \wr S_{n}/L \wr S_{n}$,
 we can analyze our stochastic model by using the theory of a Gelfand pair $(K \wr S_{n},
L \wr S_{n})$. 
\section{$N$-step probability}
We attempt to obtain the $N$-step probability, i.e., we compute $P_{(K,L;n)}^N$.
We recall a brief review of the relation between the
$N$-step probability and the theory of finite Gelfand pairs in our case.
We fix $\overline{e}=(\overline{e_{K}},\ldots,\overline{e_{K}}) \in (K/L)^n$ as an intial state.
Let $p(\overline{x},\overline{y})$ be an $(\overline{x},\overline{y})$-entry of $P_{(K,L;n)}$.
Set $\nu(\overline{x})=
p(\overline{e},\overline{x})$ and 
$\tilde{\nu}(g)=\frac{1}{\ell^nn!}\nu(\overline{g})$ ($g \in K\wr S_{n}$),
where $\ell=|L|$. Then, from Proposition \ref{trp}, $\tilde{\nu}$ is a bi-$L \wr S_{n}$ invariant function on $K\wr S_{n}$.  
Therefore, we can expand $\tilde{\nu}$ as a linear combination of $\Omega_{k}$s.
Let $\tilde{\nu}=\sum_{k \in X(s,n)}\frac{d_{k}}{|K\wr S_{n}|}f(k)\Omega_{k}$,
where $d_{k}=d_{0}^{k_{0}}\cdots d_{s-1}^{k_{s-1}}\binom{n}{k}$ is the dimension of an irreducible representation that affords to $\Omega_{k}$ (\cite{mt}). 
Let $\nu_{N}(\overline{g})$ be the $(\overline{e},\overline{g})$-entry  of $P_{(K,L;n)}^N$.
We compute
\begin{align*}
\nu_{N}(\overline{g})&=\sum_{\overline{x_{1}},\cdots,\overline{x_{N-1}}\in (K/L)^n}
p(\overline{e},\overline{x_{1}})
p(\overline{x_{1}},\overline{x_{2}})
\cdots
p(\overline{x_{N-1}},\overline{g})\\
&=
\sum_{\overline{x_{1}},\cdots,\overline{x_{N-1}}\in (K/L)^n}p(\overline{e},\overline{x_{1}})
p(\overline{e},\overline{x_{1}^{-1}x_{2}})
\cdots
p(\overline{e},\overline{x_{N-1}^{-1}g})\\
&=|L\wr S_{n}|\sum_{x_{1},\cdots,x_{N-1} \in K \wr S_{n}}\tilde{\nu}(x_{1})\tilde{\nu}(x_{1}^{-1}x_{2})\cdots \tilde{\nu}(x_{N-1}^{-1}g)\\
&=|L\wr S_{n}|\tilde{\nu}^{*N}(g) \ \ \ (N{\rm th.\ convolution\ power}).
\end{align*}
Then the idempotence of the zonal spherical functions \cite{mc} gives 
$$\nu_{N}(\overline{g})=\frac{1}{r^n}\sum_{k \in X(s,n)}{d_{k}}f(k)^{N}\Omega_{k}(g).$$ 
We call the $f(k)$'s the {\it Fourier coefficients}. 
\begin{thm}\label{fou}
For $k=(k_{0},\ldots,k_{s-1}) \in X(r,n),$ we have
$$f(k)=(1-mp)+mp\left(\sum_{i=0}^{s-1}\frac{k_{i}}{n}\overline{\omega_{i}(x_{0})}\right).$$
\end{thm}
\begin{proof}
We must compute the value of $\Omega_{k}$ evaluated at $x\in L^j\times D_{0}\times L^{n-j-1}$.
Namely, we compute the coefficient of $t^k$ in 
$\Phi_{(K,L;n)}(x)=\left(\sum_{j=0}^{s-1}t_{j}\right)^{n-1}\times \left(\sum_{j=0}^{s-1}\omega_{j}(x_{0})t_{j}\right)^1$.
Through direct computation, we obtain
$$\Omega_{k}(x)=\sum_{i=0}^{s-1}\frac{k_{i}}{n}\omega_{i}(x_{0}).$$ 
Then, the Fourier coefficient is computed as follows:
\begin{align*}
f(k)&=\sum_{g \in K \wr S_{n}}\tilde{\nu}(g)\overline{\Omega_{k}(g)}\\ 
&=\frac{1}{\ell^nn!}(1-mp)\ell^nn!
+
\frac{1}{\ell^nn!}\frac{p}{n}\sum_{j=0}^{n-1}
\sum_{g \in L^j\times D_{0}\times L^{n-j-1}}\overline{\Omega_{k}(g)}n!\\
&=(1-mp)+mp
\sum_{i=0}^{s-1}
\frac{k_{i}}{n} \overline{\omega_{i}(x_{0})}.
\end{align*}
Here, we use $m\ell=|D_{0}|=|Lx_{0}L|$ at the last equality.
\end{proof}

\section{Upper and Lower Bounds Evaluation of $\nu_{N}$}
\subsection{Upper bound evaluation}
Let $\pi$ be a uniform distribution on $(K/L)^n$.
We attempt to evaluate the total variation distance $||\nu_{N}-\pi||_{\rm TV}$ for 
some cases constructed by Definition \ref{p}.
\begin{defn}
Set $M=\max\{{\omega_{i}}(x_{0})\mid 1\leq i \leq  r-1\}$.
\end{defn}
For the remainder of this paper, we assume the 
following three conditions:
\begin{assu}\label{a2}
\begin{enumerate}
\item 
${\omega_{i}}(x_{0}) \in {\mathbb R}\  (1\leq i \leq  r-1).$ 
\item
$M<1$.
\item
$0<mp \leq1/2$.
\end{enumerate}
\end{assu}
\begin{example}
Example \ref{shut}-(1) and Example \ref{shut}-(3) satisfy Assumption \ref{a2}-(1) and (2).
\end{example} 
\begin{thm}\label{ub}
Set $N=\left[\frac{n}{2mp (1-M)}(\log{n(r-1)}+c)\right]$. Under Assumption \ref{a2}, we have
$$||\nu_{N}-\pi||_{\rm TV}\leq \ \frac{1}{4}\left(\exp\left(\exp{(-c)}\right)-1\right).$$
Here $||\nu_{N}-\pi||_{\rm TV}$ is the total variation distance between $\nu_{N}$ and the uniform 
distribution $\pi$. 
\end{thm}
\begin{proof}
We set  $k^d=k_{0}^{d_{0}}k_{1}^{d_{1}}\cdots k_{s-1}^{d_{s-1}}$ for $k \in X_{1}(s,n)$
and $d=(d_{0},\ldots,d_{s-1})$. 
 We use the following facts for our computation below:
\begin{enumerate}
\item[(a)] the upper bound lemma \cite[Corollary 4.9.2]{cst}.
\item[(b)] $d_{1}+d_{2}+\cdots+d_{s-1}=r-1$.
\item[(c)] $\exp{(-x)} \geq 1-x ,\ ({x\leq 1})$.
\end{enumerate}
Let 
  $$\psi(k_{0})= \max\left\{\left|1+mp\left(\sum_{i=1}^{s-1}\frac{k_{i}}{n}({\overline{\omega_{i}(x_{0})}}-1)
 \right)\right| \mid k_{1}+\cdots+k_{s-1}=n-k_{0}\right\}$$ for $0 \leq k_{0}\leq n-1$.
Since $mp \leq 1/2$ and $-2\leq M-1<0$, we have
$$\psi(k_{0})=
1+mp\frac{n-k_{0}}{n}(M-1).
$$
Let $X_{1}(s,n)=X(s,n)-\{(n,0,\ldots,0)\}$. We compute
\begin{align*}
||\nu_{N}-\pi||_{\rm TV}
&\underset{(a)}{\leq}
 \frac{1}{4}\sum_{k \in X_{1}(s,n)}
 d^{k}\binom{n}{k}\left|(1-mp)+mp\left(\sum_{i=0}^{s-1}\frac{k_{i}}{n}{{\omega_{i}(x_{0})}}
 \right)\right|^{2N}\notag\\
 &=
 \frac{1}{4}\sum_{k \in X_{1}(s,n)}
 d^{k}\binom{n}{k}
 \left|1+mp\left(\sum_{i=1}^{s-1}\frac{k_{i}}{n}{({\omega_{i}(x_{0})}}-1)
 \right)\right|^{2N}\notag\\
&\underset{}{\leq}
 \frac{1}{4}
 \sum_{k_{0}=0}^{n-1}
 \binom{n}{k_{0}}
 \sum_{k \in X_{}(s,n-k_{0})}
  d^{k}\binom{n-k_{0}}{k}
 \psi(k_{0})^{2N}\notag\\
&\underset{(b)}{=}
 \frac{1}{4}
 \sum_{k_{0}=0}^{n-1}
 \binom{n}{k_{0}}
(r-1)^{n-k_{0}}
\psi(k_{0})^{2N}\notag\\
&\leq
\frac{1}{4}
 \sum_{k_{0}=0}^{n-1}
\frac{n^{n-k_{0}}}{(n-k_{0})!}
(r-1)^{n-k_{0}}
\left(
1-mp\frac{n-k_{0}}{n}(1-M)
\right)^{2N}\notag\\
&\leq
\frac{1}{4}
 \sum_{k_{0}=0}^{n-1}
\frac{n^{n-k_{0}}}{(n-k_{0})!}
(r-1)^{n-k_{0}}
\exp
\left(
-2Nmp\frac{n-k_{0}}{n}(1-M)
\right)\notag\\
&\underset{(c)}{\leq}
\frac{1}{4}
 \sum_{k_{0}=1}^{n}
\frac{n^{k_{0}}}{k_{0}!}
(r-1)^{k_{0}}
\exp
\left(
-2Nmp\frac{k_{0}}{n}(1-M)
\right)\notag\\
\end{align*}
Here, we set 
$$N=\frac{n}{2mp(1-M)}(\log{n(r-1)}+c).$$
This leads to
\begin{align*}
&\frac{1}{4}\sum_{k_{0}=1}^{n}
\frac{n^{k_{0}}}{k_{0}!}
(r-1)^{k_{0}}\exp{(-2Nmp\frac{k_{0}}{n}(1-M))}\\
&\leq
\frac{1}{4}\sum_{k_{0}=1}^{\infty}
\frac{n^{k_{0}}}{k!}(r-1)^{k_{0}}\exp{(-k_{0}(\log{n(r-1)}+c))}\\
&=\frac{1}{4}\sum_{k_{0}=1}^{\infty}
\frac{\exp{(-c)}^{k_{0}}}{k!}\\
&=\frac{1}{4}\left(\exp\left(\exp{(-c)}\right)-1\right).
\end{align*}

\end{proof}
\subsection{Lower bound evaluation}
Assumption \ref{a2} is still used. 
Here our purpose  is to evaluate a lower bound of $||\nu_{N}-\pi||_{TV}$.
The main tools are Markov and  Chebyshev inequalities.
We recall the mean value of a real function $f$ on finite set $X$ for a
probability measure $\mu$ on $X$ as follows:
$$E_{\mu}(f)=\sum_{x \in X}f(x)\mu(x).$$
We also recall the variance of $f$ for $\mu$:
$${\rm V}_{\mu}(f)=E_{\mu}(f^2)-E_{\mu}(f)^2.$$
Let $A$ be a subset of $X$. Then, we have the following inequality from the definition of 
the total variation distance:
\begin{align}\label{pu}
||\nu_{N}-\pi||_{TV} \geq |\nu_{N}(A)-\pi(A)|,
\end{align}
where $\mu(A)=\sum_{a \in A}\mu(a)$ for a probability measure $\mu$.
We set
$q={\Omega_{(n-1,1,0,\ldots,0)}}$, i.e., 
\begin{align}\label{z}
q(x_{1}.\ldots,x_{n}:\sigma)=\frac{1}{n}\sum_{j=1}^{n}\omega_{1}(x_{j}).
\end{align}
We identify $q$ with a function $Q$ on $(K/L)^n$ by $Q(\overline{x})=q(x)$.
\begin{prop}
$E_{\pi}(Q)=0$.
\end{prop}
\begin{proof}
The orthogonality relation of the zonal spherical functions 
and the formula (\ref{z}) gives
\begin{align*}
E_{\pi}(Q)
&=\frac{1}{r^n}\sum_{\overline{x}\in (K/L)^n}Q(\overline{x})
=0.
\end{align*}
\end{proof}
\begin{prop}\label{Q2}
Set
$\omega_{1}^2=\sum_{j=0}^{s-1}a_{j}\omega_{j}$. Then we have
$$Q^2(\overline{x})=\frac{a_{0}}{n}+\frac{1}{n}\sum_{j=1}^{s-1}a_{j}\Omega_{(n-1,0_{j},1,0_{n-j-2})}(x)
+\left(1-\frac{1}{n}\right)
\Omega_{(n-2,2,0_{n-2})}(x).$$
Here $0_{j}=\underbrace{0,\ldots,0}_{j}$.
\end{prop}
\begin{proof}
Some properties of the zonal spherical functions give us $a_{j}\geq 0 $ $(0 \leq j \leq s-1)$, $a_{0}+\cdots+a_{s-1}=1$ (and $a_{0}=\frac{1}{d_{1}}$).
The coefficient of $t_{0}^{n-1}t_{j}$ in $\Phi_{(K,L;n)}(x)$ is given by $\sum_{i=1}^{n}\omega_{j}(x_{i})
=n\Omega_{(n-1,1,0,\ldots,0)}(x)$.
Similarly the coefficient of $t_{0}^{n-2}t_{1}^2$ in $\Phi_{(K,L;n)}(x)$ is $2\sum_{i<j}\omega_{1}(x_{i})
\omega_{1}(x_{j})=\frac{n(n-1)}{2}\Omega_{(n-2,2,0,\ldots,0)}(x)$.
Together, these give the claim of the proposition.
\end{proof}
\begin{prop}
$V_{\pi}(q)
=\frac{a_{0}}{n}$.
\end{prop}
\begin{proof}
Proposition \ref{Q2}  enable us to compute 
\begin{align*}
V_{\pi}(q)
&=E_{\pi}(Q^2)-E_{\pi}(Q)^2=E_{\pi}(Q^2)\\
&=\frac{1}{r^n}\sum_{x\in (K/L)^n}Q^2(x)=\frac{a_{0}}{n}.
\end{align*}
\end{proof}
\begin{thm}\label{lb}
Let $0< c < \log{n(r-1)}$ and $N=\left[\frac{n}{2mp(1-M)}(\log{n(r-1)}-c)
\right]$. If $n$ is sufficiently large, then
there is a constant $\delta>0$, which does not depend on $n$ such that
 $$||\nu_{N}-\pi||_{TV} \geq1-\frac{\delta}{e^c}.$$
\end{thm}
\begin{proof}
From Theorem \ref{fou}, 
the $N$-step probability is given by
$$\nu_{N}(\overline{x})=\frac{1}{r^n}\sum_{k \in X(s,n)}
d^k\binom{n}{k}\left[(1-mp)+mp\sum_{j=0}^{s-1}\frac{k_{j}}{n}\omega_{j}(x_{0})\right]^{N}
\Omega_{k}(x),$$
where $d^k=d_{0}^{k_{0}}\cdots d_{s-1}^{k_{s-1}}$.
The orthogonality relations of $\Omega_{k}$s and Proposition \ref{Q2} give us
$$E_{\mu^{*N}}(Q)=\left[1-\frac{mp}{n}(1-\omega_{1}(x_{0}))\right]^N \geq 0$$
and
\begin{align*}
E_{\nu_{N}}(Q^2)
&=
\frac{a_{0}}{n}+\sum_{j=1}^{s-1}\frac{a_{j}}{n}
\left[1-\frac{mp}{n}(1-\omega_{j}(x_{0}))\right]^N+(1-\frac{1}{n})
\left[1-\frac{2mp}{n}(1-\omega_{1}(x_{0}))\right]^N.
\end{align*}
 The following inequalities hold for $n \geq 4$:
  {\small
 \begin{align*}
V_{\nu_{N}}(Q)
&=E_{\nu_{N}}(Q^2)-E_{\nu_{N}}(Q)^2\\
&=
\sum_{j=0}^{s-1}\frac{a_{j}}{n}
\left[1-\frac{mp}{n}(1-\omega_{j}(x_{0}))\right]^N+(1-\frac{1}{n})
\left[1-\frac{2mp}{n}(1-\omega_{1}(x_{0}))\right]^N
-
\left[1-\frac{mp}{n}(1-\omega_{1}(x_{0}))\right]^{2N}\\
&\leq
\frac{a_{0}}{n}
+\frac{1-a_{0}}{n}
\left[1-\frac{mp}{n}(1-M)\right]^{N}-\frac{1}{n}
\left[1-\frac{mp}{n}(1-\omega_{1}(x_{0}))\right]^{2N}\\
& \leq
\frac{a_{0}}{n}
+\frac{1-a_{0}}{n}
\leq \frac{1}{n}.
\end{align*}
 }
%
Here we set, for $0<c<\log{n(r-1)}$, 
$$N=\frac{n}{2mp(1-M)}(\log{n(r-1)}-c).$$ 
Then, we have
 \begin{align*}
E_{\nu_{N}}(Q)
&=\left[1-\frac{mp}{n}(1-\omega_{1}(x_{0}))\right]^N
\\
&={\exp}\left[\log\left(1-\frac{mp}{n}(1-M)\right) N\right]\\
&={\exp}\left[\left(-\frac{mp}{n}(1-M)-\frac{m^2p^2}{2n^2}(1-M)^2
\varepsilon\left(\frac{mp}{n}(1-M)\right)
\right) N\right]\\
&=[n(r-1)]^{-\frac{1}{2}}e^{\frac{c}{2}}
{\exp}\left[\left(\frac{mp}{4n}{(1-M)}
(c-\log{n(r-1)})\right)
\varepsilon\left(\frac{mp}{n}(1-M)\right)
 \right]\\
  &\geq
 [n(r-1)]^{-\frac{1}{2}}e^{\frac{c}{2}}
{\exp}\left[\left(\frac{1}{8n}{(1-M)}
(c-\log{n(r-1)})\right)
\varepsilon\left(\frac{1}{2n}(1-M)\right)
 \right]
\end{align*}
Here, $\varepsilon(x)=\frac{-2(\log(1-x)+x)}{x^2}$ is a positive and monotonically increasing function on $(0,1)$ with $\lim_{x\rightarrow 0}\varepsilon(x)=1$.
Therefore, the $\exp$-section of the last formula is a monotonically increasing  function for sufficiently large $n$.
Therefore, there is $0<\gamma<1$ for sufficiently large $n$ such that 
  \begin{align*}
E_{\nu_{N}}(Q)&\geq \gamma [n(r-1)]^{-\frac{1}{2}}e^{\frac{c_{N}}{2}}
 \end{align*}
 
Set ${A_{\beta}}=\{x \in (K/L)^n \mid |Q(x)|<\frac{\beta}{\sqrt{(r-1)n}} \}$.
From Markov's inequality, we have
\begin{align*}
\pi(A_{\beta})&\geq1-\frac{n}{\beta^2}E_{\pi}(Q^2)=1-\frac{a_{0}}{\beta^2}.
 \end{align*}
 The triangle inequality gives us
 $A_{\beta}\subset \{x \in (K/L)^n\mid |Q(x)-E_{\nu_{N}}(Q)| \geq |E_{\nu_{N}}(Q)|-\frac{\beta}
 {\sqrt{(r-1)n}}\}$. Setting $\beta={\gamma e^{c/2}}/{2}$, we have
   $$\nu_{N}(A_{\beta}) \leq \frac{1/n}{(|E_{\nu_{N}}(Q)|-\beta/\sqrt{(r-1)n})^2}
   =\frac{(r-1)}{\beta^2}.$$
from Chebyshev's inequality .
 Therefore, from (\ref{pu}), we have 
 $$||\nu_{N}-\pi||_{\rm TV}\geq 1-\frac{a_{0}}{\beta^2}-\frac{(r-1)}{\beta^2}=1-\frac{4(a_{0}+(r-1))}{\gamma^2e^c}$$
 \end{proof}
We now have upper and lower bounds for $||\nu_{N}-\pi||_{\rm TV}$.
From  the definition of the cut-off (\cite[Definition 2.5.1]{cst}),
we obtain the following theorem.
 \begin{thm}
 Fix a Gelfand pair $(K,L)$ with the zonal spherical functions that  are real-valued functions. 
 Under Assumption \ref{a2}, the stochastic models defined by Definition \ref{p} have the cut-off.
 \end{thm}
 \begin{remark}
 Assumption \ref{a2}-(3) is needed for technical reasons in the paper. However, even for the $mp>1/2$ case,  similar results are expected to hold. For concrete examples that are  considered in \cite{ds} and \cite{hora1,hora2}, the authors show the cut-off by taking $mp=n/(n+1)>1/2$.
\end{remark}

\end{document}